\documentclass{article}

\usepackage{etex}
\usepackage{graphicx,amssymb}
\usepackage{color}
\usepackage{amsmath}
\usepackage{amsthm}
\usepackage{epsfig}
\usepackage{pstricks}

\usepackage{tikz}
\usepackage{soul}   
\usepackage{xy} \input xy \xyoption{all}
\usepackage{algorithm}
\usepackage{algorithmic}
\usepackage{color}
\usepackage{cancel}
\usepackage{caption}
\usepackage{subcaption}
\usepackage{enumitem}

\newcommand*{\DashedArrow}[1][]{\mathbin{\tikz [baseline=-0.25ex,-latex, dashed,#1] \draw [#1] (0pt,0.5ex) -- (1.3em,0.5ex);}}%

\newtheorem{theorem}{Theorem}[section]
\newtheorem{lemma}[theorem]{Lemma}
\newtheorem{proposition}[theorem]{Proposition}
\newtheorem{example}[theorem]{Example}
\newtheorem{remark}[theorem]{Remark}
\newtheorem{corollary}[theorem]{Corollary}

\def\qed{\hfill\vbox{\hrule\hbox{\vrule\kern3pt\vbox{\kern6pt}\kern3pt\vrule}\hrule}\bigskip}

\def\C{{\mathbb{C}}}
\def\P{{\mathbb{P}}}

\def\A{{\mathbb{A}}}
\def\K{{\mathbb{K}}}

\makeatletter
\newcommand{\dashedrightarrow}[1][2pt]{%
  \settowidth{\@tempdima}{$\longrightarrow$}\longrightarrow
  \makebox[-\@tempdima]{\hskip-1.5ex\color{white}\rule[0.5ex]{#1}{1pt}}
  \phantom{\longrightarrow}
}
\makeatother

\newcounter{mnotecounter}

\title{Covering rational surfaces with rational parametrization images}

\author{Jorge Caravantes,  J. Rafael Sendra,  David Sevilla and Carlos Villarino}

\begin{document}

\maketitle

\begin{abstract}
Let $S$ be a rational projective surface given by means of a projective rational parametrization whose base locus satisfies a mild assumption. In this paper we present an algorithm that provides three rational maps $f,g,h:\A^2\DashedArrow[->,densely dashed]S\subset \P^n$ such that the union of the three images covers $S$. As a consequence, we present a second algorithm that generates two rational maps $f,\tilde{g}:\A^2\DashedArrow[->,densely dashed]S$, such that  the union of its images  covers the affine surface $S\cap \mathbb{A}^{n}$. In the affine case, the number of rational maps involved in the cover is in general optimal.
\end{abstract}



\section{Introduction}
Rational parametrizations of algebraic varieties are an important tool in many geometric applications like those in computer aided design (see e.g. \cite{FarinHoschekKim2002a}, \cite{HL93}) or computer vision (see e.g. \cite{marsh}). Nevertheless, the applicability of this tool can be negatively affected if the parametrization is missing basic properties: for instance its injectivity, its surjectivity, or the nature of the ground field where the coefficients belong to; see e.g. the introductions of the papers \cite{SendraSevillaVillarino2015a}, \cite{SendraSevillaVillarino2014a} and \cite{SendraSevillaVillarino2016a} for some illustrating examples of this phenomenon.

In this paper we focus on the surjectivity of rational parametrizations. Surjective parametrizations, also called normal parametrizations, have been studied by many authors but still many important questions, both theoretical and computational, stay open. The case of curves, over algebraically closed fields of characteristic zero, is understood comprehensively, and one can always find a surjective, indeed also injective, affine parametrization of the  affine algebraic curve (see e.g. \cite{AndradasRecio2007a}, \cite{RSV}, \cite{Sendra2002a}, \cite{PerezSendraWinkler2008a}). Furthermore, the case of curves defined over the field of the real numbers is also studied and characterized in \cite{Sendra2002a}.

The situation, as usually happens, turns to be much more complicated when dealing with surfaces. In \cite{BajajRoyappa1995a}, \cite{GaoChou1991a} the first steps in this direction were given and answers for certain types of surfaces, like quadrics, were provided. Also, in \cite{PerezdiazSendraVillarino2010a} the relation of the polynomiality of parametrizations to the surjectivity was analyzed. Nevertheless, in \cite{MEGA_2017} it is shown that there exist rational surfaces that cannot be parametrized birationally and surjectively. As a consequence of this fact, the question of whether every rational surface can be covered by the union of the images of finitely many birational parametrizations is of interest. The answer is positive and can be deduced from the results in \cite{Bodnar2008}. In previous papers, the authors have studied this problem for special types of surfaces. In \cite{SendraSevillaVillarino2015a} the unreachable points of parametrizations of surfaces of revolution are characterized. In \cite{SendraSevillaVillarino2016a} it is proved that ruled surfaces can be covered by using two rational parametrizations. In addition, in \cite{SendraSevillaVillarino2014a} an algorithm to cover an affine rational surface without based points at infinity with at most three parametrizations is presented.

In this paper we continue the research described above and we present two main algorithmic and theoretical results. Moreover, we provide an algorithm that, for any projective surface parametrization, generates a cover of the projective surface with three parametrizations, assuming that, either the base locus of the input is empty, or the Jacobian of the input parametrization, specialized at each base point, has rank two. As a consequence of this result, we also present an algorithm that, for a given affine parametrization whose projectivization satisfies the condition on the base points mentioned above, returns a cover of the affine surface with two affine parametrizations.

Taking into account the results in \cite{MEGA_2017}, the affine cover presented in this paper is, in general, optimal. Furthermore, it improves the results in \cite{SendraSevillaVillarino2014a} and extends the results in \cite{SendraSevillaVillarino2016a} to a much more general class of surfaces. With respect to the projective cover case, although theoretically interesting, we cannot ensure that the number of parametrizations involved in the cover is optimal for a generically large class of projective surfaces since, for instance, the whole projective plane can be covered with just two maps  from the affine plane as the following example shows. We leave this theoretical question open as future reseach
\begin{example}\label{ex-intro}
Consider the following two maps:
\[\begin{array}{c@{}ccc}
f:&\A^2&\to&\P^2\\
&(s,t)&\mapsto&(1:s:t)
\end{array},\ \
\begin{array}{c@{}ccc}
g:&\A^2&\to&\P^2\\
&(s,t)&\mapsto&(s(1-s):t:(1-s)^2)
\end{array} \]
Then for any point $P:=(x:y:z)\in\P^2$, if $x\neq0$ then
$P:=f\left(\frac{y}{x},\frac{z}{x}\right)$; if $x=0\neq z$ then $P:=g\left(0,\frac{y}{z}\right)$; and if $x=z=0$ then $P:=g(1,y)$, $y\ne 0$. This means that we can cover the whole projective plane with just two maps from the affine plane.
\end{example}

The paper is structured as follows. In Section \ref{sec-preliminares}, we introduce some notation and we briefly recall some notions and results that will be used throughout the paper. In Section \ref{sec:chicha}, we present the projective cover algorithm and in Section \ref{sec-example} we illustrate the result by means of some examples. In Section \ref{sec-affine}, we apply the results in Section \ref{sec:chicha} to derive the two affine parametrization cover algorithm.

\section{Preliminaries}\label{sec-preliminares}

In this section, we briefly recall some concepts and results that will be used in the subsequent sections. We essentially recall some results on the fundamental locus of rational maps and some consequences and the characterization of zero dimensional ideals via Gr\"obner bases. Throughout this paper $\K$ is an algebraically closed field of characteristic zero, and $\P^n$ the projective space over $\K$. Moreover, we denote by $\mathbb{A}_i$ the affine space $\{(x_0:\cdots:x_n)\in \P^n,|\, x_i\neq 0\}$. In the examples, the field $\K$ will be the field $\C$ of the complex numbers.

Let $X$ be an irreducible projective variety and let $f:X\DashedArrow[->,densely dashed]\mathbb{P}^n$ be a rational map. The \emph{fundamental locus} of $f$ is the algebraic set $F(f)$ of points to which $f$ cannot be extended regularly. Any $P\in F(f)$ is called a \emph{fundamental point} of $f$. The following theorem analyzes the dimension of the fundamental locus.

\begin{theorem}\label{FundLoc}\cite[Lemma V.5.1]{Hartshorne1977a}
Let $X$ be a smooth irreducible projective variety and let $f:X\DashedArrow[->,densely dashed]\mathbb{P}^n$ be a rational map generically finite.  The fundamental locus of $f$ has codimension at least 2 in $X$.
\end{theorem}

\begin{corollary}\label{cor-case-of-surfaces}
Let $X$ be a smooth irreducible surface and $f$ as in Theorem \ref{FundLoc}. $F(f)$  is either empty or zero dimensional.
\end{corollary}

The traditional way for solving indeterminacies in algebraic geometry consists in blowing up fundamental points (see e.g.  \cite[IV.3.3]{Shafarevich_2nd_1}) and composing with the corresponding map as the next theorem shows.

\begin{theorem}\label{BlowUp}\cite[Example 7.17.3]{Hartshorne1977a} or \cite[Theorem II.7]{Beauville1996a}
Let X be a smooth surface. Let $f\colon X\DashedArrow[->,densely dashed]\mathbb{P}^n$ be a rational map. Then there exists a commutative diagram
\[ \xy
	(0,12)*+{Y}="Y";
	(-15,0)*++{X}="X";
	(15,0)*++{\mathbb{P}^n}="PN";
	{\ar_{\displaystyle g} "Y"; "X"};
	{\ar^{\displaystyle h} "Y"; "PN"};
	{\ar@{-->}^{\displaystyle f} "X"; "PN"};
\endxy \]
where $g$ is a composite of blowups involving fundamental points of $f$ and $h$ is a morphism.
\end{theorem}

A first consequence of Theorem \ref{BlowUp} is the following.

\begin{corollary}\label{cor:imagen_racional}\cite[Corollary 2.5]{MEGA_2017}
Let $X$ and $f$ be as in Theorem \ref{BlowUp}. For any fundamental point $P$ of $f$, $h(g^{-1}(P))$ is a connected finite union of rational curves.
\end{corollary}

\begin{remark}\label{rem:dimension1} Let $f\colon \mathbb{P}^2\DashedArrow[->,densely dashed]\mathbb{P}^n$ be as in Theorem \ref{BlowUp}, and let $S$ be the Zariski closure of $f(\mathbb{P}^2)$ in $\P^n$. The complementary in $S$ of the $f(\A^2)$ is, according to Theorem \ref{BlowUp}, contained in $h(g^{-1}(F(f)\cup L_\infty))$, where $L_\infty=\mathbb{P}^2-\mathbb{A}^2$. Such a subset consists of some rational curves and, if $f$ contracts $L_\infty$, a closed point (see \cite[Corollary 2.5]{MEGA_2017}).
\end{remark}

We end this section with a well-known result on elimination theory that will be used in Section \ref{sec:chicha}.

\begin{theorem}\cite[Chapter 5, Theorem 6]{CoxLittleOshea}\label{tma:forma}
Let $I$ be an ideal in $\K[x_1,..,x_n]$. Then, the following statements are equivalent:
\begin{enumerate}
\item The algebraic subset of  $\K^n$  defined by $I$ is a finite set.
\item Let $B$ be a Gr\"obner basis for $I$ with respect to a fixed monomial ordering. Then, for each $1\le i\le n$, there is some $m_i\in \mathbb{N}$ such that $x_i^{m_i}$ is the leading monomial of an element of $B$.
\end{enumerate}
\end{theorem}

\section{Covering projective surfaces with three parametrizations}\label{sec:chicha}

Throughout this section, let $S\subset \P^n$ be a rational projective surface and let
\[
F=(F_0:\cdots:F_n):\P^2\DashedArrow[->,densely dashed]S\subset\P^n
\]
be a (not necessarily birational) parametrization of $S$, given by $n+1$ homogeneous coprime polynomials $F_0,...,F_n$ where the nonzero polynomials have degree $d$. In addition, let the homogeneous ideal $I=(F_0,...,F_n)\mathbb{K}[x_0,x_1,x_2]$ be called the fundamental ideal associated to $F$. 

Since the polynomials defining $F$ are coprime, by Corollary \ref{cor-case-of-surfaces},  $I$ defines a closed algebraic subset $\mathcal{A}$ of $\P^2$ that is either empty or consists of a finite amount of points. If $\mathcal{A}=\emptyset$, then $F$ defines a regular map and its restrictions to  each of   the three affine planes $\A_i=\{(x_0:x_1:x_2)\ |\ x_i\ne0\}$ covering $\P^2$ define 3 charts that cover $S$, since the image of a projective variety by a regular map is always Zariski closed. Otherwise, say $\mathcal{A}=\{P_1,...,P_k\}$; we need to make the following assumption:
\begin{itemize}
\item[($*$)]  If $\mathcal{A}\neq \emptyset$, then for every $P\in \mathcal{A}:=\{P_1,\ldots,P_k\}$ the jacobian matrix of $F$ at $P$ has rank 2.
\end{itemize}
Note that $(*)$  guarantees that $I$ does not define multiple points (i.e. the base points of $F$ are simple). We also  assume  without loss of generality that:
\begin{itemize}
\item[$(a)$] no $P_j$ is in any of the lines $\{x_0=0\}$, $\{x_1=0\}$ and $\{x_2=0\}$,
\item[$(b)$] no pair $\{P_i$, $P_j\}$, $i\ne j$, is aligned with any of the coordinate points $(1:0:0)$, $(0:1:0)$ and $(0:0:1)$.
\end{itemize}
We observe  that  the real constraint lays in $(*)$, since conditions $(a)$ and $(b)$ are satisfied after a general change of coordinates. In the following remark we discuss how these hypotheses can be computationally checked.

\begin{remark}\label{rem:check_*} \
\begin{enumerate}
\item Note that condition $(*)$ implies that the point $P_j$ is regular in the projective scheme defined by the ideal $I$. Then, the intersection multiplicity is 1 at every $P_j$. Now, if we consider the ideal $J$ defined by the $3\times3$ minors of the jacobian matrix of $F$, the following methods, among others, can be applied to test  $(*)$:
\begin{enumerate}[label=(\roman*)]
\item  Check whether the ideal $I+J$ is zero-dimensional.
\item  Check whether  $\sqrt{I+J}=(x_0,x_1,x_2)\mathbb{K}[x_0,x_1,x_2]$ ( i.e. the irrelevant ideal) or, equivalently,  whether  $I+J$ contains a power of any of the variables.
\item  By means of resultants and gcds, using  the formulas in \cite[Theorems 2 and 3]{cox2020base} (see also \cite{Perez-Sendra2020a}).
\end{enumerate}
\item  Checking $(i)$ without explicitly determining $P_1,...,P_k$ can be carried out by certifying  that all the ideals generated by $\{F_0,...,F_n,x_i\}$ ($i=0,1,2$) either are zero-dimensional or contain a power of the irrelevant ideal.
    \item Checking $(ii)$ can be done by computing the ideal bases $B_1$ and $B_2$ of Algorithm {\tt 3PatchSurface} and checking whether they have adequate shape (see Proposition \ref{tma:deteccion_de_condiciones}). However, it may be more efficient to check that, for all $i=0,1,2$, the gcd of all resultants of couples $(F_0,F_j)$ with respect to $x_i$ is square free.
\end{enumerate}
\end{remark}

Now we consider the three affine planes $\A_i$ defined above. According to $(a)$ all $P_j$ lie in the intersection of the three affine planes. In this situation, the strategy is as follows.
We will work with the parametrization $$ f:= F|_{\A_0}:\A_0\DashedArrow[->,densely dashed]S$$ as defined in $\A_0$, and we will blowup $\A_1$ and $\A_2$ at the base points of $F$ to get new affine planes $\widetilde{\A_1}$ and $\widetilde{\A_2}$ with projections $\mathrm{Bl}_1:\widetilde{\A_1}\to\A_1$ and $\mathrm{Bl}_2:\widetilde{\A_2}\to\A_2$. Now, we introduce  the compositions $$\text{$ g:= F|_{\A_1}\circ \mathrm{Bl}_1:\widetilde{\A_1}\to\A_1\to S$ and $ h:= F|_{\A_2}\circ \mathrm{Bl}_2:\widetilde{\A_2}\to\A_2\to S$,}$$  and we prove that
the union of the images of $f$, $g$ and $h$ is the whole $S$ (see details in Proposition \ref{tma:deteccion_de_condiciones}). During this process, we also need to keep track of what happens with the infinity line of $\A_0$, namely $L_ \infty=\{(x_0:x_1:x_2)\ |\ x_0=0\}$. As a consequence, we derive the following Algorithm {\tt 3PatchSurface}.

\begin{algorithm}[h!]{{\bf Algorithm} {\tt 3PatchSurface}}

\begin{algorithmic}[1]
\REQUIRE A  map $F=(F_0:\cdots:F_n)$ defined by coprime homogeneous polynomials in $\mathbb{K}[x_0,x_1,x_2]$, where the nonzero polynomials have the same degree, parametrizing a Zariski dense subset of a projective surface $S\subset\P^n$, such that conditions $(*)$, $(a)$ and $(b)$ are satisfied.
\ENSURE Two maps $G=(G_0:\cdots:G_n)$ and $H=(H_0:\cdots:H_n)$ defined by homogeneous polynomials in $\mathbb{K}[x_0,x_1,x_2]$ where the nonzero polynomials have the same degree (while in the same list), such that the union of the images of
$F(1:\_:\_)$, $G(\_:1:\_)$, $H(\_:\_:1):\A^2\to S$ cover $S$.

\IF{the radical of the homogeneous ideal $(F_0,...,F_n)\mathbb{K}[x_0,x_1,x_2]$ is irrelevant (i.e. $F$ defines a regular morphism)}
\STATE{Return $G=F$, $H=F$.}
\ENDIF
\STATE\label{forma_de_base} For the ideals $I_i=(F_0,...,F_n,x_i-1)$ for $i=1,2$, compute reduced Gr\"obner bases \linebreak{$B_1=\{x_2-q_1(x_0),x_1-1,p_1(x_0)\}$ for $I_1$ and $B_2=\{x_2-1,x_1-q_2(x_0),p_2(x_0)\}$ for $I_2$, with lexicographical order $x_0<x_1<x_2$.}
\STATE {Set $k=\deg(p_i)>\deg(q_i)$ for whatever $i=1,2$.}
\STATE{Set $\mathbf{P}_1(x_0:x_1)=p_1(\frac{x_0}{x_1})x_1^{k}$ and $\mathbf{Q}_1(x_0:x_1)=q_1(\frac{x_0}{x_1})x_1^{k-1}$.}
\STATE\label{ExplosionA1}{Put $\widetilde{G}=(\widetilde{G}_0:\cdots:\widetilde{G}_n) =F\left(x_0x_1^{k}: x_1^{k+1}: x_1^2\mathbf{Q}_1(x_0:x_1)+x_2\mathbf{P}_1(x_0:x_1)\right)$,}
\STATE{Set $\mathbf{P}_2(x_0:x_2)=p_2(\frac{x_0}{x_2})x_2^{k}$ and $\mathbf{Q}_2(x_0:x_2)=q_2(\frac{x_0}{x_2})x_2^{k-1}$.}
\STATE{Set $\widetilde{H}=(\widetilde{H}_0:\cdots:\widetilde{H}_n)=F\left(x_0x_2^{k}: x_2^2\mathbf{Q}_2(x_0:x_2)+x_1\mathbf{P}_2(x_0:x_2): x_2^{k+1}\right)$}
\STATE \label{paso_gcd}Set $\widehat{G}=\gcd(\widetilde{G}_0,...,\widetilde{G}_n)$, $\widehat{H}=\gcd(\widetilde{H}_0,...,\widetilde{H}_n)$.
\STATE\label{paso_final} Return $G=\widetilde{G}/\widehat{G}$, $H=\widetilde{H}/\widehat{H}$.
\end{algorithmic}
\end{algorithm}

\begin{remark}
 In Proposition \ref{tma:deteccion_de_condiciones} we show   that the integer $k$, introduced in the algorithm, is exactly the number of base points, that is the cardinality of $\mathcal{A}$ (see above), so we have not introduced equivocal notation.
\end{remark}

In the following, we  see that the output of Algorithm {\tt 3PatchSurface} is correct (see Theorem \ref{tma:algoritmo_funciona}). We also recall that the required conditions for the algorithm can be checked computationally according to Remark \ref{rem:check_*}. We start by proving that Step \ref{forma_de_base} works properly, assuming that conditions $(*)$, $(a)$ and $(b)$ are satisfied. 
This is probably a well-known result in a more general setting but, up to the authors' knowledge, there are no suitable references for the proof.

\begin{proposition}\label{tma:deteccion_de_condiciones}
Let $F=(F_0:\cdots:F_n)$, $I_1$ and $I_2$ be as in Algorithm {\tt 3PatchSurface}. There exist $p_1,q_1\in\K[x_2]$, $p_2,q_2\in\K[x_1]$ such that {$k=$}$\deg(p_i)>\deg(q_i)$ for all $i=1,2$, and the reduced Gr\"obner basis $B_1$ and $B_2$ of $I_1$ and $I_2$ respectively, have the following shape:
\[B_1= \{x_2-q_1(x_0),x_1-1,p_1(x_0)\}\ \ \mbox{ and }\ \ B_2=\{x_2-1,x_1-q_2(x_0),p_2(x_0)\}.\]
\end{proposition}

\begin{proof}
Observe that both $I_1$ and $I_2$ define finite sets in $\A_1$ and $\A_2$, respectively. By Theorem \ref{tma:forma}, this implies that, since $B_1$ and $B_2$ are Gr\"obner bases, with respect lex$(x_0<x_1<x_2)$, there is a polynomial in each $B_i$ just involving $x_0$. This is $p_1$ for $I_1$ and $p_2$ for $I_2$.
Due to $(*)$, $(a)$ and $(b)$, each $p_i$  defines  $k$ different parallel lines in $\A_i$, so its degree is $k$.

Applying again Theorem \ref{tma:forma}, there is another monic polynomial in $\K[x_0][x_1]\cap I_1$. Since the basis is reduced, and $x_1-1$ was originally among the generators of $I_1$, this polynomial in $\K[x_0][x_1]$ for $I_1$ is precisely $x_1-1$. 

Now, let $q_i$ be the interpolation polynomial whose graph goes through all the base points $P_j\in \A_i$ of $F$. It does exist because $(b)$ holds (so there are no two different $P_j$ in the same vertical line) and its degree is at most $k-1$. Then $x_1-q_2(x_0)$ vanishes at every $P_j$ so, by Hilbert's Nulltellensatz, it belongs to $\sqrt{I_2}$. Since $\deg(p_2)=k>\deg(q_2)$, $x_1-q(x_0)$ cannot be reduced by $p(x_0)$, so, since $B_1$ is reduced, this is the monic polynomial in $\K[x_0][x_1]\cap I_2$.

We apply an analogous argument and reduction by $x_1-1$ to deduce that the monic polynomial in $\K[x_0,x_1][x_2]$ for $I_1$ must be $x_2-q_2(x_0)$, and we know that $x_2-1$ is a reduced monic polynomial for $I_2$, so it is in $B_2$.

%
%
Since $B_1$ and $B_2$ are reduced Gr\"obner bases with respect to the lexicographical order $x_0<x_1<x_2$ and the ideals they generate define, precisely, the fundamental locus, they are the reduced Gr\"obner bases of $I_1$ and $I_2$.
\end{proof}

Before continuing, we state a Lemma.

\begin{lemma}\label{lma:no_base_pts}
In the conditions of Algorithm {\tt 3PatchSurface}, neither $G(\_:1:\_)$ nor $H(\_:\_:1)$ have affine base points.
\end{lemma}

\begin{proof}
By the properties of the rational map $F:\P^2\DashedArrow[->,densely dashed]\P^n$, defined by $F$, and the fact that a blow up is bijective outside its blown up points, we know that any base point of $G(\_:1:\_)=F\circ\mathrm{Bl}_1(\_:1:\_)$ would be in one of the lines $\mathrm{Bl}_1(P_j)$. Now, we fix $P_j=(1:\alpha_j:\beta_j)$ and we then prove that $G(\_:1:\_)$ has no base points in the line $\{x_0=\frac{1}{\alpha_j}\}\subset\A_1$.

{%
Since all $F_i(x_0:1:x_2)$ vanish at $P_j$, they have $\alpha_jx_0-1$ as a common factor {(note that $\alpha_j\ne0$ because $(a)$ holds)}. Then, since $G_i=\widetilde{G_i}/\widehat{G}$, we have that $G_i(x_0:1:x_2)$ is a divisor of
\begin{multline*}
\overline{G_i}(x_0:1:x_2):=\frac{\widetilde{G_i}(x_0:1:x_2)}{x_0-\frac{1}{\alpha_j}}=\frac{F_i(x_0:1:q_1(x_0)+p_1(x_0)x_2)}{x_0-\frac{1}{\alpha_j}}=\\
=\frac{F_i(x_0:1:q_1(x_0)+p_1(x_0)x_2)-F_i(\frac{1}{\alpha_j}:1:q_1(\frac{1}{\alpha_j})+p_1(\frac{1}{\alpha_j})\frac{\beta}{\alpha_j})}{x_0-\frac{1}{\alpha_j}}=\\
=\frac{F_i(x_0:1:q_1(x_0)+p_1(x_0)x_2)-F_i(\frac{1}{\alpha_j}:1:q_1(\frac{1}{\alpha_j}))}{x_0-\frac{1}{\alpha_j}}.
\end{multline*}
This means that
\begin{multline*}
\overline{G_i}\left(\frac{1}{\alpha_j}:1:x_2\right) =\\
=\frac{\partial F_i}{\partial x_0}\left(\frac{1}{\alpha_j}:1:q_1\left(\frac{1}{\alpha_j}\right)\right)+\frac{\partial F_i}{\partial x_2}\left(\frac{1}{\alpha_j}:1:q_1\left(\frac{1}{\alpha_j}\right)\right)\left(\frac{\partial q_1}{\partial x_0}\left(\frac{1}{\alpha_j}\right)+\frac{\partial p_1}{\partial x_0}\left(\frac{1}{\alpha_j}\right)x_2 \right)=\\
=\left(\begin{array}{ccc}
\frac{\partial F_i}{\partial x_0}(P_j)&
\frac{\partial F_i}{\partial x_1}(P_j)&
\frac{\partial F_i}{\partial x_2}(P_j)
\end{array}\right)\cdot
\left(\begin{array}{c}
1\\0\\\frac{\partial q_1}{\partial x_0}\left(\frac{1}{\alpha_j}\right)+\frac{\partial p_1}{\partial x_0}\left(\frac{1}{\alpha_j}\right)x_2
\end{array}\right).
\end{multline*}
On the other side, the vector $(1,\alpha_j,\beta_j)$ is also in the kernel of the jacobian matrix of $F$ at $P_j$, due to Euler's formula for homogeneous polynomials: $x_0\frac{\partial F}{\partial x_0}+x_1\frac{\partial F}{\partial x_1}+x_2\frac{\partial F}{\partial x_2}=\deg(F)F$.

By $(*)$, the Jacobian matrix of $F$ has rank 2 at $P_j$, so $(1,\alpha_j,\beta_j)$ generates its kernel. Then, $\left(1,0,\frac{\partial q_1}{\partial x_0}\left(\frac{1}{\alpha_j}\right)+\frac{\partial p_1}{\partial x_0}\left(\frac{1}{\alpha_j}\right)x_2\right)$ is not in such kernel, since $\alpha_j\ne0$. Therefore,
\[\overline{G}(\frac{1}{\alpha_j}:1:x_2)=\mathrm{Jac}(F)\cdot\left(\begin{array}{c}
1\\0\\\frac{\partial q_1}{\partial x_0}\left(\frac{1}{\alpha_j}\right)+\frac{\partial p_1}{\partial x_0}\left(\frac{1}{\alpha_j}\right)x_2
\end{array}\right)\ne 0 \]
for all $x_2\in \C$. Since all entries of ${G} (\frac{1}{\alpha_j}:1:x_2 )$ are divisors of those of $\overline{G}(\frac{1}{\alpha_j}:1:x_2)$, then $G$ is nonzero throughout the whole line $x_0=\frac{1}{\alpha_j}$ in $\widetilde{\A_1}$. 
}

Repeating the argument with $H$ finishes the proof.
\end{proof}

The next result states the correctness of Algorithm {\tt 3PatchSurface}:

{
}

\begin{theorem}\label{tma:algoritmo_funciona}
The three parametrizations $F$, $G$ and $H$  output by  Algorithm {\tt 3PatchSurface} satisfy that the union of the images of $F(1:\_:\_)$, $G(\_:1:\_)$ and $H(\_:\_:1)$ covers $S$ completely.
\end{theorem}

\begin{proof}
We devote these first lines to sketch the proof. Keeping in mind Theorem \ref{BlowUp}, when $(*)$, $(a)$ and $(b)$ hold, we construct the following diagram:
\[ \xy
	(0,20)*+{\mathrm{Bl}_{P_1,...,P_k}(\P^2)}="Y";
	(-15,0)*++{\P^2}="X";
	(15,0)*++{\mathbb{P}^N}="PN";
	{\ar_{\displaystyle \pi} "Y"; "X"};
	{\ar^{\displaystyle \varphi} "Y"; "PN"};
	{\ar@{-->}^{\displaystyle F} "X"; "PN"};
	(-35,0)*++{\mathbb{A}_0}="A0";
	(-25,0)*++{\subset}="c";
	(-40,10)*++{\widetilde{\mathbb{A}_1}}="A1";
	(-42,20)*++{\widetilde{\mathbb{A}_2}}="A2";
	{\ar^{\displaystyle j_0} "A0"; "Y"};
	{\ar^{\displaystyle j_1} "A1"; "Y"};
	{\ar^{\displaystyle j_2} "A2"; "Y"};
\endxy \]
Here, $\mathrm{Bl}_{P_1,...,P_k}(\P^2)$ is the blowup of the base points. To prove that the output of Algorithm {\tt 3PatchSurface} works as expected, we need to show that
\begin{itemize}
\item $F(1:\_:\_)=\varphi\circ j_0$, $G(\_:1:\_)=\varphi\circ j_1$ and $H(\_:\_:1)=\varphi\circ j_2$.
\item $j_0(\A_0)\cup j_1(\widetilde{\A_1})\cup j_2(\widetilde{\A_2})=\mathrm{Bl}_{P_1,...,P_k}(\P^2)$.
\item $\varphi$ has no base points.
\end{itemize}
In this situation, we observe that, since $\varphi$ has no base points, then it is a regular morphism, so $\varphi(\mathrm{Bl}_{P_1,...,P_k}(\P^2))$ is an algebraic subset of $\P^n$. On the other side, $F$ is dominant over $S$, so
\[S=\varphi(\mathrm{Bl}_{P_1,...,P_k}(\P^2))=F(\A_0)\cup G(\widetilde{\A_1})\cup H(\widetilde{\A_2}),\]
and, hence, the theorem holds.
First of all, we define $j_1$ and $j_2$. The blow up of all the base points of $F$ is, locally, the blow up of the ideal $I_i$ in $\mathbb{A}_i$ with $i$ being 1 or 2. Knowing the bases of $I_1$ and $I_2$, we have {(see \cite[Section 4.2]{Shafarevich_2nd_1})}:
\[\mathrm{Bl}_{I_1}(\mathbb{A}_1)=\left\{
((x_0:1:x_2),(y_0:y_1))\ |\ p_1(x_0)y_0=(x_2-q_1(x_0))y_1
\right\}\subset\mathbb{A}_1\times\P^1,\]
\[\mathrm{Bl}_{I_2}(\mathbb{A}_2)=\left\{
((x_0:x_1:1),(y_0:y_1))\ |\ p_2(x_0)y_0=(x_1-q_2(x_0))y_1
\right\}\subset\mathbb{A}_2\times\P^1.\]
While the way to glue the two open subsets is not interesting for the purpose of this proof, it is easy to see that  $\mathrm{Bl}_{I_1}(\mathbb{A}_1)\cup \mathrm{Bl}_{I_2}(\mathbb{A}_2)$ is the whole Bl$_{\{P_1,...,P_k\}}(\P^2)$ minus the point that is the strict transform of $(1:0:0)$. We now consider the inclusions
\[\begin{array}{rccc}
j_1:&\widetilde{\mathbb{A}_1}&\to&\mathrm{Bl}_{I_1}(\mathbb{A}_1)\\
&(x_0,x_2)&\mapsto&((x_0:1:q_1(x_0)+x_2p_1(x_0)),(x_2:1))
\end{array},\]
\[\begin{array}{rccc}
j_2:&\widetilde{\mathbb{A}_2}&\to&\mathrm{Bl}_{I_2}(\mathbb{A}_2)\\
&(x_0,x_1)&\mapsto&((x_0:q_2(x_0)+x_1p_2(x_0):1),(x_1:1))
\end{array}.\]
If $\pi_i:\mathrm{Bl}_{I_i}(\mathbb{A}_i)\to \mathbb{A}_i$ is the first projection in any of the cases, it is clear that Bl$_i=\pi_i\circ j_i$ is defined by
\[\mathrm{Bl}_1(x_0,x_2)=(x_0,q_1(x_0)+x_2p_1(x_0))\mbox{ and } \mathrm{Bl}_2(x_0,x_1)=(x_0,q_2(x_0)+x_1p_2(x_0)).\]
Therefore, we have that
 \[G(x_0:1:x_2)=F(\mathrm{Bl}_1(x_0,x_2))\mbox{ and } H(x_0:x_1:1)=F(\mathrm{Bl}_2(x_0,x_1)).\]

On the other hand, $\widetilde{A_i}$ covers the whole $\A_i$ except the vertical lines through the base points. These affine lines are completely contained in $\A_0$, since the infinity point is $(0:0:1)$ for $\A_1$ and $(0:1:0)$ for $\A_2$. This means that, to show that $\widetilde{\mathbb{A}_1}$, $\widetilde{\mathbb{A}_2}$ and $\mathbb{A}_0\backslash\{P_1,...,P_k\}$ ---through $j_1$, $j_2$ and the blowup $j_0$ of the base points--- cover $\mathrm{Bl}_{\{P_1,...,P_k\}}(\P^2)$, we just need to prove that $\widetilde{\A_1}$ and $\widetilde{\A_2}$ cover the exceptional divisor. So we fix $P_j=(1:\alpha_j:\beta_j)$ and we call $E_j\simeq\P^1$ the component of the exceptional divisor corresponding to $P_j$. Note that Bl$_i$ covers a full neighborhood of $P_j$ in $\mathbb{A}_i$ minus the vertical line. This corresponds to the line joining $P_j$ with $(0:0:1)$ for the case $i=1$ and the line joining $P_j$ with $(0:1:0)$ for the case $i=2$. By condition $(b)$, these lines do not contain other base points. These two lines represent two different directions at $P_j$ (i.e. two different points in $E_j$). This means that $j_1(\{x_0=\frac{1}{\alpha_j}\})$ covers all $E_j$ except one point and that $j_2(\{x_0=\frac{1}{\beta_j}\})$ is an affine line passing through that point.

The only task remaining is proving that $\varphi$ has no base points. Such base points would be in the exceptional divisor, which is covered by $\widetilde{\A_1}$ and $\widetilde{\A_2}$, but Lemma \ref{lma:no_base_pts} states that there are no base points of $\varphi$ in $j_1(\widetilde{\A_1})\cup j_2(\widetilde{\A_2})$.
\end{proof}

\section{Examples}\label{sec-example}

This section is devoted to illustrating Algorithm {\tt 3PatchSurface} by examples. We start with a toy example in which we explicitly show that the three parametrizations cover the whole projective surface.

\begin{example}\label{ex:whitney}
Whitney's umbrella has implicit equation $y_0y_1^2-y_2^2y_3=0$. The usual parametrization $(x_0:x_1:x_2)\mapsto(x_0^2:x_1x_2:x_0x_1:x_2^2)$ does not satisfy condition $(a)$, so we change coordinates to get a new parametrization
\begin{multline*}
F(x_0:x_1:x_2)=(x_2^2 - 2x_0x_2 + x_0^2:
x_1^2 + x_1x_2 - x_0x_1 - x_0x_2:\\
-x_1x_2 + x_0x_1 + x_0x_2 - x_0^2:
x_1^2 + 2x_1x_2 + x_2^2).
\end{multline*}
The only base point of the new parametrization is $(1:-1:1)$. We then compute the bases
\[B_1=\{x_2+1,x_1-1,x_0+1\}\mbox{ and }
B_2=\{x_2-1,x_1+1,x_0-1\}.\]
So we get:
\[\widetilde{G}(x_0:x_1:x_2)=(x_0+x_1)G(x_0:x_1:x_2)\mbox{ and }\widetilde{H}(x_0:x_1:x_2)=(x_2-x_0)H(x_0:x_1:x_2)\]
where
\begin{multline*}
G=
(x_1^3 - 2x_1^2x_2 + x_1^2x_0 + x_1x_2^2 - 2x_1x_2x_0 + x_2^2x_0:
x_1^2x_2 - x_1x_2x_0:\\
x_1^3 - x_1^2x_2 - x_1^2x_0 + x_1x_2x_0: x_1x_2^2 + x_2^2x_0)
\end{multline*}
and
\begin{multline*}
H=
(x_2^3 - x_2^2x_0:
x_1^2x_2 - x_1^2x_0 + x_1x_2^2 + x_0x_1x_2:
x_1x_2^2 - x_0x_1x_2 + x_2^3 + x_0x_2^2:
x_1^2x_2 - x_0x_1^2).
\end{multline*}

We now prove that the whole surface is covered by $F(1:\_:\_)$, $G(\_:1:\_)$ and $H(\_:\_:1)$. Let $A=(y_0:y_1:y_2:y_3)$ be such that $y_0y_1^2-y_2^2y_3=0$. Then:
\begin{enumerate}
\item if $y_0y_2( y_0y_1+y_0y_2-y_2^2)\ne0$, then, taking $y_0=1$,  
\[A=F\left(1: \frac{y_2+y_1+y_2^2}{y_2+y_1-y_2^2}: \frac{-y_2+y_1-y_2^2}{y_2+y_1-y_2^2}\right)\]
Note here that $y_3=\frac{y_1^2}{y_2^2}$, due to the equation of the umbrella and $y_0=1$.
\item if $y_0y_1+y_0y_2-y_2^2=0$ and $y_0y_2\ne0$, taking $y_0=1$, we get $y_1=y_2^2-y_2$. This equality transforms the equation of the umbrella in $y_3=(y_2-1)^2$. Then, {$A=G(0:1:1-1/y_2)$}.
\item if $y_2=0$ and $y_0\ne0$, then,  taking $y_0=1$,
\[A=F\left(1:1:\frac{-1+\sqrt{y_3}}{1+\sqrt{y_3}}\right).\]
{Note here that every $A$ in this case is gotten twice. Even $(1:0:0:1)$ is both $F(1:1:0)$ and $H(0:-1:1)$.}
\item if $y_0=0$, the infinity hyperplane section of Whitney's umbrella has two components: $y_2=0$ and $y_3=0$. Then:
\[A=(0:y_1:0:y_3)=F\left(1:0:\frac{(y_3+y_1)}{(y_3-y_1)}\right)\]
when $y_3- y_1\ne0$, and $A=(0:1:0:1)=G(0:1:1)$. Moreover,
\[A=(0:y_1:y_2:0)=G\left(-1:y_1+y_2:y_1\right)\]
when $y_1+y_2\ne0$. Finally, $A=(0:-1:1:0)=H(1:-1:1)$.
\end{enumerate}
Observe that $F(1:\_:\_)$ covers the whole umbrella minus a couple of rational curves. Then, $G(\_:1:\_)$ covers these curves minus just a point, that is covered by $H(\_:\_:1)$.
\end{example}

The following example applies Algorithm {\tt 3PatchSurface} to a classic surface. While the computation time is not long, the output is too large, so we just sketch some computations.

\begin{example}\label{ex:cubica}
Let us consider the Clebsch cubic, given by the equation
\[z_0^3+z_1^3+z_2^3+z_3^3-(z_0+z_1+z_2+z_3)^3=0.\]
 A parametrization is defined by
\begin{multline*}
(- x_0^2x_1+x_0^2x_2+x_0x_1^2-x_0x_2^2:
x_0^3  - x_0^2x_1- x_0^2x_2 + x_1x_2^2:\\
- x_0^3+ x_0^2x_1+ x_0^2x_2-x_1^2x_2:
- x_0x_1^2 + x_0^2x_1-x_1x_2^2 ).
\end{multline*}
This parametrization, however, does not satisfy conditions $(a)$ and $(b)$. After the coordinate change in $\P^2$ given by
\[(x_0:x_1:x_2)\mapsto(x_0+3x_1+2x_2:x_0+x_1+3x_2:-x_0-x_1+x_2),\]
we get the parametrization $F=(F_0:F_1:F_2:F_3)$, where
\[
F_0=8x_2^3 + 8x_2^2x_1 + 8x_2^2x_0 - 18x_2x_1^2 - 12x_2x_1x_0 - 2x_2x_0^2 - 18x_1^3 - 30x_1^2x_0 - 14x_1x_0^2 - 2x_0^3,\]
\[F_1= -5x_2^3 - 17x_2^2x_1 - 9x_2^2x_0 + 19x_2x_1^2 + 14x_2x_1x_0 + 3x_2x_0^2 + 28x_1^3 + 30x_1^2x_0 + 12x_1x_0^2 + 2x_0^3,\]
\[F_2= -x_2^3 + 15x_2^2x_1 + 7x_2^2x_0 - 13x_2x_1^2 - 2x_2x_1x_0 + 3x_2x_0^2 - 26x_1^3 - 24x_1^2x_0 - 6x_1x_0^2,\]
\[F_3= -9x_2^3 + 6x_2^2x_1 + 18x_2x_1^2 + 4x_2x_1x_0 - 2x_2x_0^2 + 5x_1^3 + 5x_1^2x_0 - x_1x_0^2 - x_0^3.
\]
One can check that the base points are $P_1=(5:2:-3)$, $P_2=(11:-2:-5)$, $P_3=(7:-2:-1)$, $P_4=(1:2:1)$, $P_{5,6}=(\pm2-\sqrt{5}:2:\sqrt{5})$. This agrees with the well known fact that a cubic smooth surface is a plane blown up at 6 general points.

Now we take $x_1=1$ or $x_2=1$ and get the two Gr\"obner bases $B_i=\{x_i-1,x_{3-i}-q_i(x_0),p_i(x_0)\}$, where $i\in\{1,2\}$ and:
\[p_1(x_0)=x_0^6 + 8x_0^5 + \frac{21}{4}x_0^4 - 61x_0^3 - \frac{1077}{16}x_0^2 + \frac{239}{4}x_0 - \frac{385}{64},\]
\[q_1(x_0)=-\frac{393}{8360}x_0^5 - \frac{18511}{50160}x_0^4 - \frac{12941}{50160}x_0^3 + \frac{13537}{5280}x_0^2 + \frac{1123141}{401280}x_0 - \frac{21649}{14592},\]
\[p_2(x_0)=x_0^6 + \frac{178}{15}x_0^5 + \frac{199}{5}x_0^4 + \frac{916}{25}x_0^3 - \frac{2387}{75}x_0^2 - \frac{3926}{75}x_0 - \frac{77}{15},\]
\[q_2(x_0)=-\frac{7075}{11264}x_0^5 - \frac{212195}{33792}x_0^4 - \frac{221885}{16896}x_0^3 + \frac{49613}{16896}x_0^2 + \frac{612691}{33792}x_0 + \frac{2987}{3072}.\]
Performing substitutions
\[\widetilde{G}(x_0:1:x_2) = F\left(x_0:1:q_1(x_0)+x_2p_1(x_0)\right)\]
and
\[\widetilde{H}(x_0:x_1:1) = F\left(x_0:q_2(x_0)+x_1p_2(x_0):1\right)\]
then dividing by the gcd of all entries in each case produces two   parametrizations of degree 15, $G(x_0:1:x_2)$ and $H(x_0:x_1:1)$, with about $45$ coefficients per polynomial and about $70$ bits per coefficient,  that cover, together with $F(1:x_1:x_2)$, the whole cubic.
\end{example}

\section{The affine case}\label{sec-affine}

In this section, we slightly change our point of view and we consider the problem of covering a rational affine surface by means of the images of several affine parametrizations. So, in the sequel we consider that we are given $F$ and $S$ as in Section \ref{sec-affine}, and we deal with the problem of covering  $S\cap\A^n$, where $\A^n$ is the open subset of $\P^n$ defined by the first variable not vanishing. Equivalently, one may consider that we are indeed given an affine parametrization and the affine surface that it defines. Nevertheless, to be consistent with the notation used throughout the paper, we will use the first notational statement of the problem.

In this section, we prove that to cover a rational affine surface, only two patches are necessary (see Theorem \ref{tma:algoritmo_afin_chufla} and Corollary \ref{cor:1_parche}). The basic idea is as follows. The given parametrization $F(1:x_1:x_2)$ covers a constructible subset. The complement of such subset is contained in the image of $L_\infty$ (that is, $F(0:x_1:x_2)$) and the base points, which is a finite union of affine rational curves (see Corollary \ref{cor:imagen_racional}) and, maybe, an isolated point corresponding to a contracted $L_\infty$.  The parametrization $G(x_0:1:x_2)$ of Algorithm {\tt 3PatchSurface}, restricted to certain vertical lines, covers all such affine curves except at most one point. Since such curves also have a point at infinity, we want such point to be the image of the point at infinity of the parameter line.  Based on these ideas, we derive  Algorithm {\tt 2PatchForAffine} that, when the original parametrization satisfies $(*)$, $(a)$ and $(b)$, covers affine surfaces using just two parametrizations.	

\begin{algorithm}{{\bf Algorithm} {\tt 2PatchForAffine}}
\begin{algorithmic}[1]
\REQUIRE A list $F=(F_0:\cdots:F_n)$ of coprime homogeneous polynomials of the same degree in $\mathbb{K}[x_0,x_1,x_2]$, parametrizing a Zariski dense subset of a projective surface $S\subset\P^n_\C$, such that conditions $(*)$, $(a)$ and $(b)$ are satisfied.
\ENSURE A list $g'=(g_0':\cdots:g_n')$ of rational functions of two variables such that
$F(1:\_:\_)$ and $g'(\_,\_)$ cover $S\cap\A^n$ .
\STATE Compute $B_1=\{x_2-q_1(x_0),x_1-1,p_1(x_0)\}$ and $G=(G_0:\cdots:G_n)$ as in Algorithm $\mathtt{3PatchSurface}$.
\FOR{$\alpha$ root of $p_1(x_0)$ \OR $\alpha=0$}
  \IF {deg $G_0(\alpha:1:x_2)<$ max$\{$deg $G_i(\alpha:1:x_2)\ |\ i=1,...,n\}$ \OR $G_0(\alpha:1:x_2)$ is constant}
    \STATE Include $\alpha$ in set $A$ and $\beta_\alpha:=\infty$. 
    \ELSE
    \STATE Include $\alpha$ in set $B$. Choose $\beta_\alpha$ among the roots of $G_0(\alpha:1:x_2)$.
  \ENDIF
\ENDFOR
\STATE Let $s(x_0)$ be a polynomial vanishing at all the $\alpha\in  A$ and not vanishing at any of the $\alpha\in B$. See Remark \ref{rem:calcular_s} below, for suggestions on how to find one.
\STATE Choose $r(x_0)$, a polynomial such that $r(\alpha)=\beta_\alpha s(\alpha)$ for all $\alpha\in B$ and $r(\alpha)\ne0$ for all $\alpha\in A$.
\STATE Find Bezout coefficients $u$ and $v$ such that  $\gcd(r,s)=u\cdot r+v \cdot s$.
\RETURN $g'(x_0,x_2)=G\left(x_0:1:\dfrac{r(x_0)x_2+v(x_0)}{s(x_0)x_2-u(x_0)}\right)$. 
\end{algorithmic}
\end{algorithm}

\begin{remark}\label{rem:calcular_s} \ Let us comment some computational aspects of Algorithm {\tt 2PatchForAffine}.
\begin{enumerate}
\item For $s(x_0)$, one may proceed as follows: collect the coefficients of $G_0(x_0:1:x_2)$ for $x_2$ except the one for $x_2^0$, and compute the gcd, $d(x_0)$, of these coefficients; then $s(x_0)=\gcd(d(x_0),x_0 p_1(x_0))$.
\item
Note that Algorithm {\tt 3PatchSurface} does not extend the field that is used to define $F$, However, Algorithm {\tt 2PatchForAffine} needs to consider possibly algebraic coordinates for the points that need to be sent to the infinity in the parameter plane $\A_1$. Example \ref{ex:2parches} shows that the application of the algorithm forces the usage of algebraic coefficients.
\end{enumerate}
\end{remark}

In order to state the correctness of Algorithm {\tt 2PatchForAffine}, we start with a technical lemma.

\begin{lemma}\label{lma:G(beta=infty}
In the setting of Algorithm {\tt 2PatchForAffine}, if no component of $G(\alpha:1:x_2)$ is constant, it holds that
\begin{enumerate}
\item If $\alpha\in A$, then $\lim_{x_2\to\infty}G(\alpha:1:x_2)$ is a point at infinity.
\item If $\alpha\in B$, then $G(\alpha:1:\beta_\alpha)$ is a point at infinity.
\end{enumerate}
\end{lemma}

\begin{proof}
Let $\alpha$ be an element of $A$. If $G_0(\alpha:1:x_0)$ is identically zero, then the result is obvious. Otherwise, the map   $(\frac{G_1}{G_0}(\alpha:1:x_2),...,\frac{G_n}{G_0}(\alpha:1:x_2))$, from the affine plane to the surface, has degree strictly higher in the numerator of at least one of its components, because $\deg(G_0(\alpha:1:x_2))<\max\{\deg(G_i(\alpha:1:x_2))\ |\ i=1,...,n\}$. So the limit, when $x_2$ tends to $\infty$, is at infinity. Note that the full $(n+1)-$tuple $G(\alpha:1:x_2)$ is not constant, so the case of constant first entry $G_0(\alpha:1:x_2)$ satisfies the inequality too.

Now, let $\alpha$ be an element of $B$. Then, $G_0(\alpha:1:\beta_\alpha)=0$, so the statement holds.
\end{proof}

Let us, now, prove that Algorithm {\tt 2PatchForAffine} works as expected.

\begin{theorem}\label{tma:algoritmo_afin_chufla}
Let $F$, $S$ and $g'$ be as in Algorithm {\tt 2PatchForAffine}. Then $F(1:x_1:x_2)$ and $g'(x_0,x_2)$ cover the whole $S\cap\{y_0\ne0\}$.
\end{theorem}

\begin{proof}
Since the input of Algorithm {\tt 2PatchForAffine} and  Algorithm {\tt 3PatchSurface} are the same, by Theorem \ref{tma:algoritmo_funciona}, one deduces that  $F(1:x_1:x_2)$, $G(x_0:1:x_2)$ and $H(x_0:x_1:1)$ cover the projective surface $S$.

Taking into account how $G$ and $H$ in Algorithm {\tt 3PatchSurface} are defined, any point in $\widetilde{\A_1}$ and $\widetilde{\A_2}$, not in $L_\infty$ or in $\{p_i(x_0)=0\}$, is sent  by $G(x_0:1:x_2)$ or $H(x_0:x_1:1)$ into the image by $F$ of a point in $\A_0$. Therefore, we need to check that $g'$ covers any affine point in $G(\{x_0=0\}\cup\{p_1(x_0)=0\})\cup H(\{x_0=0\}\cup\{p_2(x_0)=0\})$.

Any component $C$ of $G(\{x_0=0\}\cup\{p_1(x_0)=0\})\cup H(\{x_0=0\}\cup\{p_2(x_0)=0\})$ is either a point or a rational curve covered by $G(\alpha:1:x_2)\cup H(\alpha':x_1:1)$, where either $\alpha=\alpha'=0$ or $(1:\frac{1}{\alpha}:\frac{1}{\alpha'})$ is a base point of $F$. Moreover, such component is the Zariski closure of the image of the restriction $G|_{\{x_0=\alpha,x_1\ne0\}}$, which coincides with the Zariski closure of $g'|_{\{x_0=\alpha\}}$.

If $C$ is just a point, then it is covered by $g'|_{\{x_0=\alpha\}}$. Otherwise, it is well known that any morphism defined in an open subset of a projective smooth curve can be extended regularly to the whole curve and that the image of a projective curve by a regular morphism is a Zariski closed subset. Then, we can extend $g'|_{\{x_0=\alpha\}}$ to the Zariski closure of the line where it is defined and we would cover completely $C$. This means that $g'|_{\{x_0=\alpha\}}$ covers all $C$  minus, at most, just a point (the image of the infinity point of the affine line). However, this point is $G(\alpha:1:\beta_\alpha)$, which is at infinity by Lemma \ref{lma:G(beta=infty}. Therefore, any point in $C\cap\{y_0\ne0\}$ is in $g'(\{x_0=\alpha\})\subset g'(\A^2)$.
\end{proof}

\begin{example}\label{ex:2parches}
Consider the following projective transformation of the Veronese morphism:
\[\begin{array}{c@{}ccc}
F:&\P^2&\to&\P^5\\
&(x_0:x_1:x_2)&\mapsto& (x_0^2+x_1^2+x_2^2:x_0x_1: x_0x_2:x_1^2:x_1x_2:x_2^2)
\end{array} \]
Since there are no base points, Algorithm {\tt 2PatchForAffine} generates $G=F$, and then, for just $\alpha=0$, it computes $\beta_0$ such that $\beta_0^2+1=0$. This means that $\beta_0$ must be imaginary, so it is not rational. The output
\[g'(x_0,x_2)=G\left(x_0:1:\frac{ix_2+1}{x_2}\right)\]
has imaginary coefficients.

We observe that any choice of two rationally defined affine planes $\A_0$ and $\A_1$ of the projective plane will leave a point $P$ in the projective plane over $\mathbb{Q}$ out of the union, and then $F(P)$, which is not at infinity, will not be covered. Observe, however, that the surface is isomorphic to the projective plane, so one can compose the two maps appearing in Example \ref{ex-intro} with the Veronese morphism to cover, not just the affine part, but the whole projective surface without extending the field.
\end{example}

\begin{example}\label{ex:cubica_afin_2_parches}
Let us again consider the cubic of Example \ref{ex:cubica}. We recall that
\[p_1(x_0)=x_0^6 + 8x_0^5 + \frac{21}{4}x_0^4 - 61x_0^3 - \frac{1077}{16}x_0^2 + \frac{239}{4}x_0 - \frac{385}{64}.\]
One can factor $G_0$, as obtained in Example \ref{ex:cubica}, and one of the factors is
\begin{multline*}
200640x_2x_0^4 + 1203840x_2x_0^3 - 1304160x_2x_0^2 - 9329760x_2x_0 +\\
4827900x_2 - 9432x_0^3 - 55180x_0^2 + 56238x_0 + 388135.
\end{multline*}
Here, one can easily get $x_2$ as a rational function on $x_0$, so we have a point going to infinity at each vertical line:
\begin{equation}\label{eq:beta_alfa}
x_2=\beta_{x_0}=\frac{9432x_0^3 + 55180x_0^2 - 56238x_0 - 388135}{200640x_0^4 + 1203840x_0^3 - 1304160x_0^2 - 9329760x_0 + 4827900}.
\end{equation}
The set $A$ is given by the common roots of $p_1$ and the denominator in \eqref{eq:beta_alfa}, so
\[s(x_0)=\gcd(p_1,\mathrm{denominator}(\beta_{x_0}))
=16x_0^4+96x_0^3-104x_0^2-744x_0+385.\]
We need a polynomial $r(x_0)$, coprime with $s(x_0)$, whose values at the roots of $\frac{x_0p_1(x_0)}{s(x_0)}$ equal $\beta_{x_0}s(x_0)$. Such roots are $0$ and $1\pm\frac{\sqrt{5}}{2}$, and the interpolating polynomial
\[r(x_0)=\frac{18511}{3135}x_0^2 - \frac{898}{209}x_0 - \frac{7057}{228}\]
works. We now need a Bezout identity $ur+vs=1$, so we get
\begin{multline*}
u(x_0)=-\frac{26198287461}{21693258568709}x_0^3 - \frac{1732666485971}{130159551412254}x_0^2 -\\
\frac{487041584557}{12396147753548}x_0 - \frac{632041996387}{74376886521288}
\end{multline*}
and
\[v(x_0)=\frac{1510767910251}{3389825077278640}x_0 + \frac{357074564524303}{186537231395390304}.\]
Then, the images of $F(1:x_1:x_2)$ and $G\left(x_0:1:\dfrac{r(x_0)x_2+v(x_0)}{s(x_0)x_2-u(x_0)}\right)$ cover the whole affine cubic.
\end{example}

In \cite{MEGA_2017}, it is proved that there exist affine surfaces that cannot be covered by means of a unique map from the affine plane. In fact, the surface in Example \ref{ex:cubica_afin_2_parches} is proved to be one of them. Now, the following  corollary of Theorem \ref{tma:algoritmo_afin_chufla} shows that, under hypotheses (*), (a), (b), one can always cover the affine surface with two affine parametrization images.

\begin{corollary}\label{cor:1_parche}
Let $S$ be an affine surface such that there exists a parametrization $f:\A^2\DashedArrow[->,densely dashed]S$ with a projectivization $F$ satisfying $(*)$, $(a)$ and $(b)$. Then $S$ can be covered with just two parametrizations.
\end{corollary}

In order to prove that two affine patches are enough we have had to impose, to the projectivization of the input affine parametrization, hypotheses (*), (a) and (b). If we do not impose (*), we cannot ensure this general result. However, it is interesting to observe that there are affine surfaces not satisfying $(*)$ that can be covered by only one map. To create an example, it is enough to send the exceptional divisor to the infinity hyperplane together with the image of $L_\infty$.

\begin{example}
Consider the transformation of the plane  $t(x_0:x_1:x_2)=(x_0x_1x_2:x_1^3:x_0x_2^2)$ and let $F$ be its composition with the degree-3 Veronese morphism $v_3(x_0:x_1:x_2)= (x_0x_1x_2:x_0^3:\cdots:x_2^3)$. We can observe that $\P^2-t(\A_0)=\{x_0x_1x_2=0\}=v_3^{-1}(\{y_0=0\})$. On the other hand, the fundamental locus is defined by the ideal $I=(x_0x_1x_2,x_1^3,x_0x_2^2)$, which is singular, so $F$ does not satisfy $(*)$, but $F(1,x_1,x_2)$ covers the whole affine surface.
\end{example}

\section{Conclusions}

Covering a rational surface as the union of images of several parametrizations is an important problem with direct implications in the potential feasibility of the applications. There were previous results for some particular types of surfaces. In this paper we enlarge the class of surfaces where this approach is valid. More precisely, we present two algorithms, one for the projective case and another for the affine case. In the latter, we are able to cover any rational affine surface, satisfying certain hypotheses on the base locus of the input parametrization, that is optimal in the number of cover elements, namely, two. For the projective case, the answer provides three cover parametrizations. Some open problems are, on one hand, the extension of the results to the case where no condition of the base locus is imposed, and the optimality on the number of cover parametrizations in the projective case.

\noindent \textbf{Acknowledgments.}
The authors are partially supported by FEDER/Ministerio de Ciencia, Innovaci\'on y Universidades - Agencia Estatal de Investigaci\'on/MTM2017-88796-P (Symbolic Computation: new challenges in Algebra and Geometry together with its applications). J. Caravantes, J.R. Sendra and C. Villarino belong to the Research Group ASYNACS (Ref. CT-CE2019/683). D. Sevilla is a member of the research group GADAC and is partially supported by Junta de Extremadura and FEDER funds (group FQM024).

\bibliography{3patches_arxiv}
\bibliographystyle{apalike}

\noindent Jorge Caravantes, Research Group ASYNACS \\ Dpto. de F\'isica y Matem\'aticas, Universidad de Alcal\'a \\ 28871 Alcal\'a de Henares (Madrid), Spain \\ Email: \texttt{jorge.caravantes@uah.es}

\medskip

\noindent J. Rafael Sendra, Research Group ASYNACS \\ Dpto. de F\'isica y Matem\'aticas, Universidad de Alcal\'a \\ 28871 Alcal\'a de Henares (Madrid), Spain \\
E-mail: \texttt{Rafael.sendra@uah.es}

\medskip

\noindent David Sevilla, Research group GADAC \\ Centro U. de M\'erida, Universidad de Extremadura \\ Av. Santa Teresa de Jornet 38, 06800 M\'erida (Badajoz), Spain \\
E-mail: \texttt{sevillad@unex.es}

\medskip

\noindent Carlos Villarino, Research Group ASYNACS \\ Dpto. de F\'isica y Matem\'aticas, Universidad de Alcal\'a \\ 28871 Alcal\'a de Henares (Madrid), Spain \\
E-mail: \texttt{Carlos.villarino@uah.es}

\end{document}